\documentclass{article}
\usepackage{amsmath,amssymb,amsthm}
\usepackage{authblk}
\usepackage{graphicx}
\usepackage{natbib}
\usepackage[a4paper,left=3cm,right=3cm,top=3cm,bottom=3cm]{geometry}

\author[1,3]{Patrick Rubin-Delanchy}
\author[2,3]{Nicholas A Heard}
\author[4]{Daniel J Lawson}

\affil[1]{Department of Statistics, University of Oxford, UK}
\affil[2]{Department of Mathematics, Imperial College London, UK}
\affil[3]{Heilbronn Institute for Mathematical Research, University of Bristol, UK}
\affil[4]{School of Social and Community Medicine, University of Bristol, UK}

\def\Prob{\mathrm{P}}
\DeclareMathOperator*{\var}{var}
\DeclareMathOperator*{\sgn}{sgn}

\def\N{\mathbb{N}}

\def\R{\mathbb{R}}

\def\E{\mathrm{E}}

\def\R{\mathbb{R}}

\newtheorem{theorem}{Theorem}
\newtheorem{lemma}{Lemma}

\date{}
\title{Meta-analysis of mid-p-values: some new results based on the convex order}

\begin{document}
\maketitle
\begin{abstract}
The mid-p-value is a proposed improvement on the ordinary p-value for the case where the test statistic is partially or completely discrete. In this case, the ordinary p-value is conservative, meaning that its null distribution is larger than a uniform distribution on the unit interval, in the usual stochastic order. The mid-p-value is not conservative. However, its null distribution is dominated by the uniform distribution in a different stochastic order, called the convex order. The property leads us to discover some new finite-sample and asymptotic bounds on functions of mid-p-values, which can be used to combine results from different hypothesis tests conservatively, yet more powerfully, using mid-p-values rather than p-values. Our methodology is demonstrated on real data from a cyber-security application. 
\end{abstract}
\noindent%
\vfill

\section{Introduction}
\label{sec:intro}
Let $T$ be a real-valued test statistic, with probability measure $\Prob_0$ under the null hypothesis, denoted $H_0$. Let $X$ be a uniform random variable on the unit interval that is independent of $T$ under $\Prob_0$. $X$ is a randomisation device which is in practice usually generated by a computer. 

We consider the (one-sided) p-value,
\begin{equation}
P = \Prob_0(T^* \geq T), \label{eq:p-value}
\end{equation}
the mid-p-value  \citep{lancaster1952statistical},
\begin{equation}
Q = \frac{1}{2} \Prob_0(T^* \geq T)+\frac{1}{2} \Prob_0(T^* > T), \label{eq:q-value}
\end{equation}
and the randomised p-value,
\begin{equation}
R = X \Prob_0(T^* \geq T)+(1-X)\Prob_0(T^* > T), \label{eq:r-value}
\end{equation}
where $T^*$ is a hypothetical independent replicate of $T$ under $\Prob_0$. If $T$ is absolutely continuous under $H_0$, then the three quantities are equal and distributed uniformly on the unit interval. More generally, that is, if discrete components are possible, the three are different. Two main factors, one obvious and one more subtle, make this a very common occurrence. First, $T$ is discrete if it is a function of discrete data, e.g. a contingency table, categorical data or a presence/absence event. Second, discrete test statistics often occur as a result of conditioning, as in the permutation test or Kendall's tau test \citep{sheskin2003handbook}. Partially discrete tests occur, for example, as a result of censoring.

When $P$, $Q$ and $R$ are not equal, it is a question which to choose. The ordinary p-value is often preferred in relatively strict hypothesis testing conditions, e.g. in clinical trials, where the probability of rejecting the null hypothesis must not exceed the nominal level (often 5\%). The randomised p-value has some theoretical advantages, e.g. the nominal level of the test is met exactly. However, to quote one of its earliest proponents, ``most people will find repugnant the idea of adding yet
another random element to a result which is already subject to the errors of random sampling'' \citep{stevens1950fiducial}. Randomised p-values also fail Birnbaum's admissibility criterion \citep{birnbaum1954combining}. Note that we can also work with an unrealised version of the randomised p-value, known as the \emph{fuzzy} or \emph{abstract} p-value \citep{geyer2005fuzzy}, and either stop there --- leaving interpretation to the decision-maker --- or propagate uncertainty through to any post-hoc analysis, e.g. multiple-testing \citep{kulinskaya2009fuzzy,habiger2015multiple}.

Although it can allow breaches of the nominal level, the mid-p-value is often deemed to better represent the evidence against the null hypothesis than the ordinary or randomised p-values. Justifications are not just heuristic as, for example, the mid-p-value can arise as a Rao-Blackwellisation of the randomised p-value corresponding to the uniformly most powerful test \citep{wells2010optimality}, as an optimal estimate of the $H_0$ versus $H_1$ truth indicator under squared loss \citep{hwang2001optimality}, or from asymptotic Bayesian arguments \citep{routledge1994practicing}. Performance has also been demonstrated in applications, e.g. in the context of healthcare monitoring \citep{spiegelhalter2012statistical} (a paper read before the Royal Statistical Society), genetics \citep{graffelman2013mid}, a wealth of examples involving contingency tables \citep{lydersen2009recommended}, and more. Our own interest stems from cyber-security applications, and a motivating example is given in Section~\ref{sec:intrusion_detection}. Most arguments for using the mid-p-value in hypothesis testing scenarios also work for confidence intervals. Here, using the mid-p-value over the p-value can result in a smaller interval, with a closer-to-nominal coverage probability \citep{berry1995mid,fagerland2015recommended}. 

In this article, we are able to make further mathematical progress on the mid-p-value by using a stochastic order known as the \emph{convex order}. The problem we focus on is meta-analysis, that is, combining evidence from different hypothesis tests into one, global measure of significance. In some of the scenarios analysed, the use of the ordinary p-value leads to sub-optimal, and even spurious results. New bounds for some commonly-used methods for combining ordinary p-values are derived for mid-p-values. This allows large gains in power over using ordinary p-values, while, unlike any previous study based on mid-p-values, the false positive rate is still controlled exactly (albeit conservatively).

The remainder of this article is structured as follows. In Section~\ref{sec:main_results}, we summarise our main results. Section~\ref{sec:intrusion_detection} gives a cyber-security application where, using mid-p-values, we are able to detect a cyber-attack that would likely fall under the radar if only ordinary p-values were used. Section~\ref{sec:combining} elaborates on the results of Section~\ref{sec:intrusion_detection}, with improved (although more complicated) bounds, simulations and discussion. Section~\ref{sec:conclusion} concludes. All proofs are relegated to the Appendix.

\section{Main results}\label{sec:main_results}
This section summarises the main ideas and findings of the paper. Let $U$ denote a uniform random variable on the unit interval, with expectation operator $\E$, and let $\E_0$ denote expectation with respect to $\Prob_0$. Under the null hypothesis, it is well known, see e.g. \citet{casella2002statistical}, that $P$ dominates $U$ in the \emph{usual stochastic order}, denoted $P \geq_{st} U$. One way to write this is
\begin{equation}
\E_0\{f(P)\} \geq \E\{f(U)\},\label{eq:conservative}
\end{equation}
for any non-decreasing function $f$, whenever the expectations exist \citep{shaked07}. It is also well known, and in fact true by design, that $R$ is uniformly distributed under the null hypothesis, denoted $R =_{st} U$. On the other hand, it is not widely known that, under the null hypothesis, $Q$ is dominated by $U$ in the \emph{convex order}, denoted $Q \leq_{cx} U$. One way to write this is  \citep[Chapter 3]{shaked07}
\begin{equation}
\E_{0}\{h(Q)\} \leq \E\{h(U)\}, \label{eq:convex}
\end{equation}
for any \emph{convex} function $h$, whenever the expectations exist. We have used the qualifier `widely', because an effective equivalent of equation \eqref{eq:convex} can be found in \citet{hwang2001optimality}. However, even there, equation \eqref{eq:convex} is not recognised as a major stochastic order, meaning that some of its importance is missed.

In particular, we now present three concrete, new results, made possible by the literature on the convex order. Each provides a method for combining mid-p-values conservatively, the first two in finite samples and the last asymptotically. Details and improved (but more complicated) bounds are given in Section~\ref{sec:combining}. In what follows, $Q_1, \ldots, Q_n$ denote independent (but not necessarily identically distributed) mid-p-values, with an implied joint probability measure $\tilde \Prob_0$ under the null hypothesis. 

Let $\bar Q_n = n^{-1} \sum_{i=1}^n Q_i$ denote the average mid-p-value. For $t \geq 0$,
\begin{equation}
\tilde \Prob_0\left(1/2 - \bar Q_n\geq t\right) \leq \exp(-6 n t^2). \label{eq:sum_bound}
\end{equation}
Note that, first, no knowledge of the individual mid-p-value distributions is required. Second, Hoeffding's inequality \citep{hoeffding1963probability}, which would be available more generally, gives the larger bound $\exp(-2nt^2)$ (the cubic root).

Let $F_n = -2 \sum_{i=1}^n \log(Q_i)$, known as Fisher's statistic \citep{fisher1934statistical} and the most popular method for combining p-values (in the continuous case, it is well-known that $F_n$ has a chi-square distribution with $2n$ degrees of freedom under $H_0$). For $t \geq 2n$,
\begin{equation}
\tilde \Prob_0(F_n \geq t) \leq \exp\{n-t/2-n\log(2 n/t)\}. \label{eq:fisher_bound}
\end{equation}
Finally, assume additionally that $Q_1, \ldots, Q_n$ are identically distributed. Then applying Fisher's method as usual, i.e. treating the mid-p-values as if they were ordinary p-values and using the chi-square tail, is asymptotically conservative as $n \rightarrow \infty$. 

\section{Example: network intrusion detection}\label{sec:intrusion_detection}
The perceived importance of cyber-security research has risen dramatically in recent years, particularly after several well-publicised events in 2016 and 2017. In this area, anomaly detection over very high volumes and rates of network data is a key statistical problem \citep{adams2016dynamic}. In our experience of the field, discrete data, whether they be presence/absence events, counts or categorical data, are absolutely the norm rather than the exception. We will demonstrate the value of our paper's contributions in a network intrusion detection problem. 

Figure \ref{fig:auth} shows publically available authentication data covering 58 days on the Los Alamos National Laboratory computer network \citep{kent16}. Nodes in the graph are computers, and an edge indicates that there was at least one connection from one computer to the other, resulting in a graph with $m \approx 18,000$ nodes and $~\sim400,000$ directed edges. An exciting opportunity offered by this data resource is that it contains an actual cyber-attack: or, to be precise, records of penetration testing activity conducted by a `red-team'. One of the four computers used for the attack (the highest degree of the four, ID ``C17693'', with 296 out of 534 edges labelled as nefarious) is highlighted in red on the left, with its connections highlighted in pink on the right. 

\begin{figure}[t]
  \centering
  \includegraphics[width=13cm]{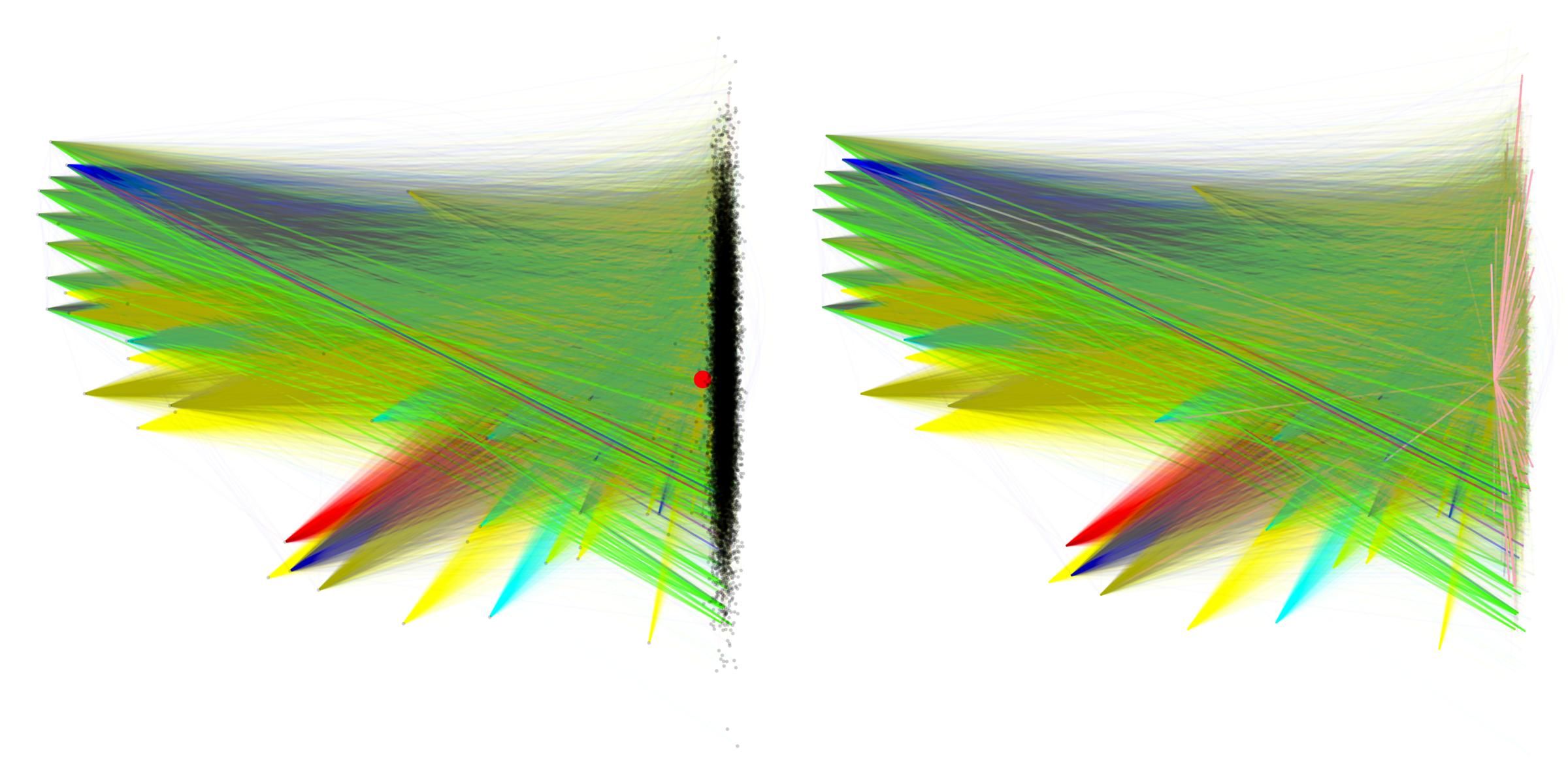}
  \caption{Authentication data: full network of connections comprising $\sim18,000$ nodes and $~\sim400,000$ directed edges. Edges are coloured by authentication type. On the left, nodes are shown as black points, with node ID ``C17693'' highlighted in red (and larger). On the right, the points are hidden to better see the connections made by node ID ``C17693'', which are now highlighted in pink.}
  \label{fig:auth}
\end{figure}

Earlier work on network intrusion has suggested that the occurrence of \emph{new edges} on the network can be indicative of nefarious behaviour \citep{neil2013towards,neil2015using}. Looking at the outward connections from a given computer, in particular, those which involve a computer otherwise receiving relatively few new connections present special interest. Because the first day of data has no red-team activity, we use this day to learn a rate $\lambda_j, j = 1, \ldots, m$ at which each computer receives new connections, treating the times as right-censored independent and identically distributed exponential random variables. For every computer on the network, the set of outward new connections made over the remainder of the observation period $[1,58]$ is scored according to this model. The test-statistic 
\[T_{ij} = \begin{cases} 57 & \text{if no connection occurs from $i$ to $j$,} \\ \tau -1 & \text{if a new connection from $i$ to $j$ occurs at time $\tau$}, \end{cases}\]
is considered for every directed pair $(i,j)$ not occurring as an edge on the first day, so that each node $i$ has associated with it a collection of test statistics $T_{i\cdot}$, which are partially discrete, with a point mass at 57.

For regularisation purposes, the rates $\lambda_j, j = 1, \ldots, m$ are assumed \emph{a priori} to follow a Gamma distribution matching the mean and variance of the empirical rates computed for each $j = 1, \ldots, m$ over the full period of 58 days. The use of this prior implies that before censoring $T_{ij}$ has a Gamma-Exponential (also called Lomax) predictive distribution, which is used to compute the collection of ordinary, mid, and randomised p-values $P_{i\cdot}, Q_{i\cdot}, R_{i\cdot}$ corresponding to the outward connections of each node $i = 1, \ldots, m$.

As we are interested in the \emph{ranking} of computer ID ``C17693'' among the other $\sim18,000$ computers, as well as its p-value, it makes sense to extend the ranges of the bounds \eqref{eq:sum_bound} and \eqref{eq:fisher_bound} as follows:
\begin{align}
\tilde \Prob_0\left(1/2 - \bar Q_n\geq t\right) &\leq \exp\{-6 \sgn(t) n t^2\}, & t \in \R,\label{eq:sum_bound2}\\
\tilde \Prob_0(F_n \geq t) &\leq \exp[\sgn(t-2n)\{n-t/2-n\log(2 n/t)\}], & t>0, \label{eq:fisher_bound2}
\end{align}
which preserves the shape and monotonicity of the curves, and remains valid because larger values than unity are returned outside the old ranges. Our options are:
\begin{enumerate}
\item to compute the average ordinary, mid, and randomised p-values, and obtain a global significance level using bound \eqref{eq:sum_bound2}. Computer ID ``C17693'' then ranks as 8th (p-value $\approx 1$), 8th (p-value $\approx 10^{-7}$) and 9th (p-value $\approx 10^{-7}$) most anomalous of the $\sim 18,000$ computers respectively. 
\item to compute Fisher's statistic for the ordinary, mid, and randomised p-values, and obtain a global significance level using bound \eqref{eq:fisher_bound2} for the second case, and the chi-square tail otherwise. Computer ID ``C17693'' now ranks joint 8118th (p-value $\approx 1$), 2nd (p-value $\approx 1$) and 9th (p-value $\approx 10^{-43}$) respectively.
\item to assume an asymptotic regime and use the chi-square tail for the Fisher-with-mid-p-values statistic instead. Computer ID ``C17693'' then ranks 8th (p-value $\approx 1$).
\end{enumerate}

As rankings go, therefore, the mid-p-value is never beaten, with computer ID ``C17693'' coming in the top ten every time and coming second once. The most obvious approach of using Fisher's method with ordinary p-values fails completely. As for the other three red-team computers: using the best performing method, i.e. Fisher's statistic with mid-p-values and bound \eqref{eq:fisher_bound2}, where Computer ID ``C17693'' comes second, their ranks are 384th (ID ``C18025''), 550th (ID ``C19932'') and 1079th (ID ``C22409'').

\section{Meta-analysis of mid-p-values: further details}\label{sec:combining}
This section elaborates on the results of Section~\ref{sec:main_results}. We say that a random variable (and its measure and distribution function) is \emph{sub-uniform} if it is less variable than a uniform random variable, $U$, in the convex order. 

To see why the mid-p-value is sub-uniform, notice that $Q = \E_0(R \mid T)$. By Jensen's inequality, for any convex function $h$, 
\begin{equation}
\E_0\{h(Q)\} = \E_0[h\{\E_0(R\mid T)\}] \leq \E_0[\E_0\{h(R) \mid T\}] = \E_0\{h(R)\} = \E\{h(U)\}, \label{eq:proof}
\end{equation}
whenever the expectations exist, since $R =_{st} U$. Remember that we do not claim this result is new, see e.g. \citet{hwang2001optimality}, but rather the idea to exploit the literature on the convex order.

To formalise the meta-analysis framework, let $T_1, \ldots, T_n$ be a sequence of independent test statistics. We consider a joint null hypothesis, $\tilde H_0$, under which $T_1, \ldots, T_n$ have probability measure $\Prob_0^{(1)}, \ldots, \Prob_0^{(n)}$ respectively. The p-values, $P_i$, mid-p-values, $Q_i$, and randomised p-values, $R_i$, are obtained by replacing $\Prob_0$ with $\Prob_0^{(i)}$ in \eqref{eq:p-value}, \eqref{eq:q-value} and \eqref{eq:r-value} respectively. In the case of the randomised p-value, an independent uniform variable, $X_i$, is generated each time. $\tilde \Prob_0$ denotes the implied joint probability measure of the statistics under $\tilde H_0$. The focus of this section is on testing the joint null hypothesis $\tilde H_0$.
Probability bounds that follow often have the form $\tilde \Prob_0\{ f(Q_1, \ldots, Q_n) \geq t\} \leq b_n(t)$. If the observed mid-p-values are $q_1, \ldots, q_n$ and level of the test is $\alpha$ (e.g. 5\%), then a procedure that rejects when  $b_n\{f(q_1, \ldots, q_n)\}\leq \alpha$ is conservative: the probability of rejecting $\tilde H_0$ if $\tilde H_0$ is true does not exceed $\alpha$.

\subsection{Sums of mid-p-values}\label{sec:sums}
An early advocate of mid-p-values, \citet{barnard1989alleged, barnard1990must} proposed to combine test results from different contingency tables by taking the sum of standardised mid-p-values. His exposition relies on some approximations. Our results make exact inference possible.

We begin with a bound on the sum of independent mid-p-values. This bound bears an interesting resemblance to Hoeffding's inequality \citep{hoeffding1963probability}. It will later be extended to be relevant to Barnard's analysis.

\begin{theorem}\label{thm:sums}
Let $X_1, \ldots, X_n$ denote $n$ independent sub-uniform random variables with mean $\bar X_n=n^{-1} \sum_{i=1}^n X_i$. Then, for $0 \leq t \leq 1/2$,
\begin{align}
\Prob\left(1/2 - \bar X_n\geq t\right) &\leq \min_{h\geq 0}\left\{2 e^{-h t} \sinh(h/2)/h\right\}^n,\label{eq:bound1}\\
& \leq \exp(-12 n t^2) \left\{\sinh(6 t)/(6t)\right\}^n, \label{eq:bound2}\\
& \leq \exp(-6 n t^2). \label{eq:bound}
\end{align}
\end{theorem}
A sub-uniform random variable has expectation $1/2$ and is bounded between $0$ and $1$. Hoeffding's inequality would therefore give us $\Prob\left(1/2 - \bar X_n\geq t\right) \leq \exp(-2 n t^2)$, the cubic root. Our improvement is substantial, for example, suppose we observe an average of 0.4 from $n=100$ mid-p-values. This is very significant: $\tilde \Prob_0\left(1/2 - \bar Q_n\geq 0.1\right) \leq 0.0025$ using \eqref{eq:bound}. However, we would only find $\tilde \Prob_0\left(1/2 - \bar Q_n\geq 0.1\right) \leq 0.14$ using Hoeffding's inequality.

Instead of summing the mid-p-values directly, \citet{barnard1990must} actually considers sums of the standardised statistics
\[D_i = (1/2 - Q_i)/\sigma_i,\]
where $\sigma_i$ is the standard deviation of $Q_i$ under $\tilde H_0$. The upper tail probability of the sum is then estimated by Gaussian approximation. In the purely discrete case, Barnard shows that $\sigma_i = \{(1-s_i)/12\}^{1/2}$ where
\[s_i = \sum_{t \in S_i} \left\{\Prob^{(i)}_0(T_i = t)\right\}^{3},\]
and $S_i$ is the (countable) support of $Q_i$. Instead of appealing to the Gaussian approximation, the convex order allows us to find an exact bound.

\begin{lemma}
\label{thm:bound2}
Let $X_1, \ldots, X_n$ denote $n$ independent sub-uniform random variables with standard deviations $\sigma_1, \ldots, \sigma_n$ respectively, and let 
\[\bar Y_n = \frac{1}{n} \sum_{i=1}^n (1/2 - X_i)/\sigma_i.\]
Then, for $t \geq 0$,
\begin{align}
\Prob(\bar Y_n \geq t) &\leq \min_{h\geq0} \left(\prod_{i=1}^n \exp[-h\{t+1/(2 \sigma_i)\}] \left\{ \frac{e^{h/\sigma_i} - 1}{h/\sigma_i} + h^2 \left(\frac{1}{2}- \frac{1}{24\sigma_i^2}\right)\right\}\right),  \label{eq:bestbound}\\
&\leq \exp\{ - 6 n (\bar \sigma t)^2\} \label{eq:secondbestbound},
\end{align}
where $\bar \sigma  = (\prod \sigma_i)^{1/n}$ is the geometric mean of the standard deviations.
\end{lemma}

In practice, the bound \eqref{eq:bestbound}, which is an important improvement over \eqref{eq:secondbestbound}, is found numerically by minimising over $h$. Of course, even if the optimum cannot be determined exactly the obtained bound still holds, because the tail area is simply over-estimated.

To illustrate how the bound \eqref{eq:bestbound} performs in practice, we now re-visit Barnard's example \citep[p.606]{barnard1990must}. The first experiment he considers yields $Q_1 = 1/7, s_1 = 9002/42^3, D_1 = 1.32$. The second yields $Q_2=1/9, s_2=141/729, D_2 = 1.5$. Since the sum divided by $\sqrt{2}$ is almost two, i.e. two standard deviations away, he finds ``serious evidence'' against the null hypothesis. Lemma \ref{thm:bound2} finds $\tilde \Prob_0(D_1 + D_2 \geq 1.32+1.5) \leq 0.12$, providing some evidence in favour of the alternative, but not significant at, say, the $5\%$ level. On the other hand, evidence would start to become compelling if we were to observe the second result again,  $Q_3=1/9, s_3=141/729, D_3 = 1.5$; Lemma \ref{thm:bound2} then finds $\tilde \Prob_0(D_1 + D_2 + D_3 \geq 1.32+1.5+1.5) \leq 0.036$. 

\subsection{Products of mid-p-values (Fisher's method)}
Fisher's method \citep{fisher1934statistical} is the most popular way of combining p-values. As is well-known, under $\tilde H_0$, the statistic $-2\sum_{i=1}^n \log(P_i)$ has a chi-square distribution with $2n$ degrees of freedom if $P_i$ are absolutely continuous. Therefore, the p-value of the combined test is $P^\dagger = S_{2n}\{-2\sum_{i=1}^n \log(P_i)\}$, where $S_k$ is the survival function of a chi-square distribution with $k$ degrees of freedom. This results in an exact procedure when $P_i$ are absolutely continuous, and a conservative one otherwise, i.e. $P^\dagger \geq_{st} U$ under $\tilde H_0$. 

Our next result allows us to use the mid-p-values $Q_1, \ldots, Q_n$ in place of $P_1, \ldots, P_n$ while retaining a conservative procedure. We were able to derive three probability bounds. None beats the other two uniformly for all $n$ and all significance levels (see Figure~\ref{fig:fisher_bound}), but the last is often the winner, hence the simpler statement of Section~\ref{sec:main_results}.

\begin{theorem}\label{thm:fisher_finite}
Let $X_1, \ldots, X_n$ be a sequence of independent sub-uniform random variables. Then for $x \geq 2 n$, 
\begin{multline*}
\Prob\left(-2 \sum_{i=1}^n \log(X_i) \geq x\right) \leq \min\Big[S_{2m}(x - 2 n \log 2),\\
n\left/\left[n+\{(x-2n)/2\}^2\right]\right., \exp\{n-x/2-n\log(2 n/x)\}\Big] = u_n(x).
\end{multline*}
\end{theorem}

The first uses $\Prob(X_i \leq \alpha) \leq 2 \alpha$ for $\alpha \geq 0$, which would be obvious if $X_i$ was a mid-p-value, but is actually true for any sub-uniform random variable \citep{meng94}. The second uses bounds on the mean and variance of $-\log(X_i)$ (given in Lemma \ref{lem:mean_and_var}, in the Appendix) and then applies the Chebyshev-Cantelli inequality. The third is based on a bound on the moment generating function of $-\log(X_i)$. Derivation details are in the Appendix. 

For a given $n$ and $\alpha \in (0,1]$, let $t_{\alpha,n}$ denote the critical value of Fisher's statistic, i.e., $t_{\alpha,n}$ satisfies $S_{2n}(t_{\alpha,n}) = \alpha$. Figure \ref{fig:fisher_bound} presents the behaviour of the different bounds for different $n$ ($20$ on the left and 1 billion on the right) and $\alpha$. The curves show the bound given by each formula at different $\alpha$ (which can be interpreted as `canonical levels'), i.e. inputting $x = t_{\alpha,n}$ in Theorem~\ref{thm:fisher_finite}, as $\alpha$ ranges from $10^{-5}$ to $0.1$. For low $\alpha$, the bound based on the moment generating function, marked MGF, is by far the best. 

\begin{figure}[t]
  \centering
  \includegraphics[width=13cm]{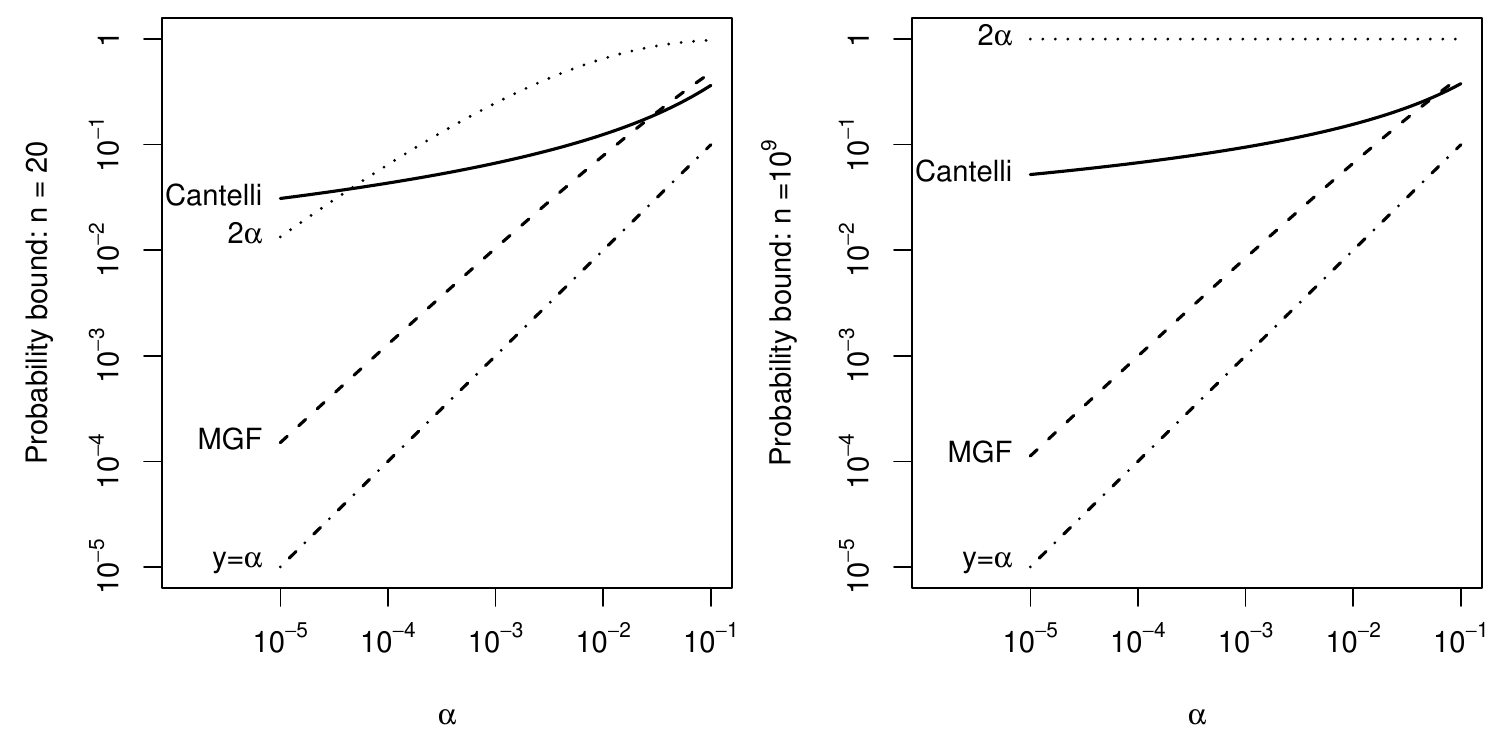}
  \caption{Comparison of the probability bounds given by Theorem~\ref{thm:fisher_finite} for Fisher's method using mid-p-values. Theorem~\ref{thm:fisher_finite} gives explicit formulae for $2 \alpha$, Cantelli and MGF, in that order. Both axes are on the logarithmic scale.}\label{fig:fisher_bound}
\end{figure}

Let $Q^\dagger = u_n\{-2 \sum_{i=1}^n \log(Q_i)\}$. Then $Q^\dagger$ is again conservative, i.e., $Q^\dagger \geq_{st} U$ under $\tilde H_0$. Both $P^\dagger$ and $Q^\dagger$ are valid p-values. Clearly, if the underlying p-values are continuous then the standard $P^\dagger$ is superior (in fact, deterministically smaller). However, $Q^\dagger$ seems to be substantially more powerful in a wide range of discrete cases. This is demonstrated by simulation in Section~\ref{sec:simulations}. 

Finally, we find this interesting asymptotic result.
\begin{theorem}[Fisher's method is asymptotically conservative]\label{thm:fisher_asymp}
Let $X_1, \ldots, X_n$ denote $n$ independent and identically distributed sub-uniform random variables. For any $\alpha \in (0,1]$, there exists  $N \in \N$ such that 
\[\Prob\left(- 2 \sum_{i=1}^n \log(X_i) \geq t_{\alpha, n}\right) \leq \alpha,\]
for any $n \geq N$.
\end{theorem}
Hence, we can dispense with any correction entirely if $n$ is large enough and the $Q_i$ are identically distributed. A formal proof is given in the Appendix. Since $\E\{-\log(X_i)\} \leq \E\{-\log(U)\}$, from the definition of the convex order, a direct application of the law of large numbers gets us most of way, except for the possibility $\E\{-\log(X_i)\} = \E\{-\log(U)\}$. In fact, this exception is no problem because, perhaps surprisingly, it implies that the $X_i$ are uniform, using \citet[Theorem 3.A.43]{shaked07}.

\subsection{Simulations}\label{sec:simulations}
To illustrate the potential improvement of employing Fisher's method with mid-p-values, using the bound \eqref{eq:fisher_bound}, over the traditional approach of using ordinary p-values and the chi-square tail, we considered p-values from three types of support. In the first column, each p-value $P_i$ can only take one of two values, $1/2$ and $1$. We therefore have $Q_i = 0.25$ if $P_i = 1/2$ and $Q_i = 0.75$ if $P_i = 1$. Under the null hypothesis, $\Prob^{(i)}_0(P_i =1/2) = \Prob^{(i)}_0(P_i=1)=1/2$. In the second column, each p-value $P_i$ is supported on the pair $\{p_i, 1\}$, where $p_i$ is drawn uniformly on the unit interval. We therefore have $Q_i = p_i/2$ if $P_i = p_i$ and $Q_i = (1+p_i)/2$ otherwise. Under the null hypothesis, $\Prob^{(i)}_0(P_i =p_i) = 1-\Prob^{(i)}_0(P_i=1)=p_i$, for each $i$. Finally, in the third column each p-value $P_i$ takes one of ten values, $1/10, 2/10, \ldots, 1$, and therefore $Q_i = P_i - 1/20$. Under the null hypothesis, $\Prob^{(i)}_0(P_i = j/10) = 1/10$, for $j = 1, \ldots, 10$. The rows represent two different alternatives and sample sizes. In both cases, the $P_i$ are generated by left-censoring a sequence of independent and identically distributed Beta variables, $B_1, \ldots, B_n$, that is, $P_i$ is the smallest supported value larger than $B_i$. In the first scenario, the dataset is small ($n=10$), but the signal is strong (a Beta distribution with parameters 1 and 20). In the second the dataset is larger ($n=100$) but the signal is made weaker accordingly (a Beta distribution with parameters 1 and 20). Comparing just the solid and dashed lines first, we see that $Q^\dagger$ always outperforms $P^\dagger$ substantially, and sometimes overwhelmingly. In the bottom-left corner, for example, we have a situation where, at a false positive rate set to $5\%$ say, the test $Q^\dagger$ would detect the effect with probability close to one whereas with $P^\dagger$ the probability would be close to zero. 

As a final possibility, consider $R^\dagger = S_{2n}\{-2\sum_{i=1}^n \log(R_i)\}$. A disappointment is that this randomised version, the dotted line in Figure ~\ref{fig:fisher}, tends to outperform even the mid-p-values, and by a substantial margin. On the other hand, as pointed out in the introduction, the randomised p-value has some important philosophical disadvantages, and did not perform better in our real data example.

\begin{figure}
  \centering
  \includegraphics[width=12cm]{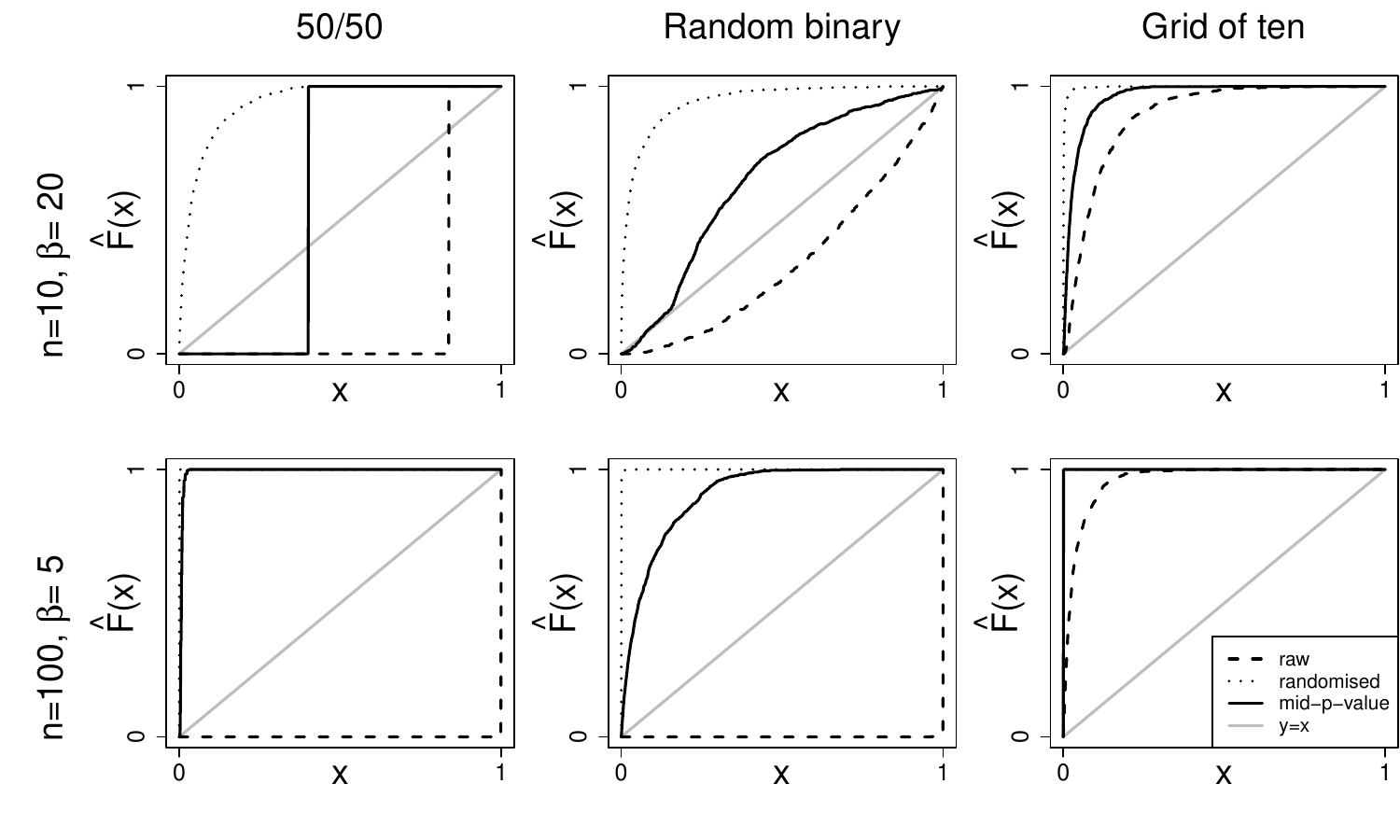}
  \caption{Fisher's method with discrete p-values. Empirical distribution functions of Fisher's combined p-value under different conditions. 50/50: each p-value is equal to 1/2 or 1 (with probability 1/2 each under $\tilde H_0$). Random binary: each p-value is equal to $p$ or 1 (with probability $p$ and $1-p$ respectively under $\tilde H_0$). $p$ is drawn uniformly on $[0,1]$ (independently of whether $\tilde H_0$ or $\tilde H_1$ holds). Grid of ten: each p-value is drawn from $1/10, 2/10 \ldots, 1$ (with probability $1/10$ each under $\tilde H_0$). $n=10, \beta=20$: 10 p-values from a left-censored Beta$(1,20)$ distribution. $n=100, \beta=5$: 100 p-values from a left-censored Beta$(1,5)$ distribution. Dotted line: randomised p-values. Solid line: mid-p-value. Dashed line: standard p-values. Further details in main text.} \label{fig:fisher}
\end{figure}

\section{Conclusion}\label{sec:conclusion}
The convex order provides a formal platform for the treatment and interpretation of mid-p-values. This article used mathematical results from this literature to combine mid-p-values, which are not conservative individually, into an overall significance level that is conservative. As shown in real data and simulations, the gains in power can be substantial.

Whereas the focus of this article was on meta-analysis, another canonical problem is multiple testing, where the task is to subselect from or adjust a set of p-values, for example, subject to a maximum false discovery rate \citep{benjamini1995controlling}. The case of discrete data has been analysed in a number of papers, including \citet{kulinskaya2009fuzzy,habiger2011randomised,liang2015false,habiger2015multiple}. A promising (but ostensibly harder) avenue of research would be to investigate the use of the convex order in this problem. 

\appendix
\section*{Appendix}
\begin{proof}[Proof of Theorem~\ref{thm:sums}]
Since $1-X$ is sub-uniform if and only if $X$ is sub-uniform, it is sufficient to prove the bounds in \eqref{eq:bound1}, \eqref{eq:bound2} and \eqref{eq:bound} hold for $\Prob\left(\bar X_n-1/2\geq t\right)$. Since $\exp(x h)$ is a convex function in $x$ for any $h$, the convex order gives us
$\E\{\exp(h X_i)\} \leq \E\{\exp(h U)\} = (e^h-1)/h$. Therefore, for any $h \geq 0$,
\begin{align*}
\Prob\left(\bar X_n -1/2 \geq t\right) &= \Prob\left[\exp\left(\sum_{i=1}^n h X_i\right) 
\geq \exp\{n h (t+1/2)\}\right],\\
&\leq \exp\{-n h (t+1/2)\}\E\left\{\exp\left(\sum_{i=1}^n h X_i\right)\right\} ,\\
& \leq \exp\{-n h (t+1/2)\}\{(e^h-1)/h\}^n\\
& = \left\{2 e^{-h t} \sinh(h/2)/h\right\}^n,
\end{align*}
where the second line follows from Markov's inequality. The choice $h = 12 t$ (motivated by an analysis of the Taylor expansion in $h$ at $0$) leads to
\begin{align*}
\Prob\left(\bar X_n -1/2 \geq t\right) &\leq \exp(-12 n t^2) \left\{\sinh(6 t)/(6t)\right\}^n\\
&\leq \exp(-6 n t^2) \left\{e^{-6 t} \sinh(6 t)/(6t)\right\}^n
\leq \exp(-6 n t^2).
\end{align*}
using the fact that $e^{-x}\sinh(x)/x = (1-e^{-2x})/(2x)$ is one at $x=0$ (using l'Hospital's rule) and decreasing. 
\end{proof}

\begin{proof}[Proof of Lemma \ref{thm:bound2}]
Again, we will prove the bound holds for $W_n= n^{-1}\sum (X_i-1/2)/\sigma_i$, so that the theorem holds by symmetry. For any $h \geq 0$,
\begin{align*}
\E\{\exp(hX_i/\sigma_i)\} & = 1+\E(h X_i/\sigma_i) + \E\left\{(hX_i/\sigma_i)^2\right\}/2 + \ldots\\
&= 1+\E(h U /\sigma_i)+ h^2\left(\frac{1}{2}+\frac{1}{8 \sigma_i^2}\right)  + \ldots \\
&\leq \E\{\exp(h U/\sigma_i)\} + h^2\left(\frac{1}{2}+\frac{1}{8 \sigma_i^2} - \frac{1}{6 \sigma_i^2}\right),
\end{align*}
because $\E\{(hX_i/\sigma_i)^n\} \leq \E\{(hU/\sigma_i)^n\}$ for $n \geq 3$, by the convex order, and $\E\{(U/\sigma_i)^2\}/2=1/(6\sigma_i^2)$.  Therefore,
\begin{align*}
\Prob(W_n \geq t) &= \Prob\left[\exp\left\{\sum_{i=1}^n h (X_i - 1/2)/\sigma_i\right\} \geq e^{h n t}\right],\\
&\leq e^{-hnt}\E\left[\exp\left\{\sum_{i=1}^n h (X_i - 1/2)/\sigma_i\right\}\right],\\
&=\prod_{i=1}^n \exp[-h\{t+1/(2 \sigma_i)\}] \left\{ \frac{e^{h/\sigma_i} - 1}{h/\sigma_i} + h^2 \left(\frac{1}{2}- \frac{1}{24\sigma_i^2}\right)\right\},
\end{align*}
proving that \eqref{eq:bestbound} holds. Next, since $\sigma^2_i \leq 1/12$, 
\begin{align*}
\Prob(W_n \geq t) &\leq \prod_{i=1}^n \exp[-h\{t+1/(2 \sigma_i)\}] \left( \frac{e^{h/\sigma_i} - 1}{h/\sigma_i}\right)\\
& =  \left(\left.2 e^{- h t} \left[\prod_{i=1}^n \sinh\{h/(2\sigma_i)\}\right]^{1/n}\right/(h/\bar{\sigma})\right)^n\\
& \leq \left\{2 e^{-h t} \sinh(h/(2 \bar \sigma))/(h/\bar \sigma)\right\}^n,
\end{align*}
using the fact that the function $\sinh$ is geometrically convex on $[0,\infty)$ \citep{niculescu2000convexity}. We proceed as in the proof of Theorem \ref{thm:sums}, choosing $h = 12\bar\sigma t$.
\end{proof}

The proofs of Theorems \ref{thm:fisher_finite} and \ref{thm:fisher_asymp} both need the following result. 
\begin{lemma}\label{lem:mean_and_var}
Let $X$ be a sub-uniform random variable. Then either i) $X$ is uniform on $[0,1]$  or ii)
\[\E\{-\log(X)\} < \E\{-\log(U)\} = 1; \quad \var\{-\log(X)\} < \var\{-\log(U)\} = 1, \]
where $U$ is a uniform random variable on $[0,1]$
\end{lemma}

\begin{proof}
\citet[Theorem 3.A.43]{shaked07} provide the following theorem. If $X \leq_{cx} Y$ and for some strictly convex function $h$ we have $\E\{h(X)\} = \E\{h(Y)\}$ then $X$ is distributed as $Y$. The function $-\log(x)$ is strictly convex, therefore either $X$ is uniform or $\E\{-\log(X)\} < \E\{-\log(U)\}$. If the latter is true, then
\begin{align*}
\var\{-\log(X)\} &= \E[-\log(X) - \E\{-\log(X)\}]^2 \\
&< \E[-\log(X) - \E\{-\log(U)\}]^2\\
&\leq \E\{\log(U) + 1\}^2\\
&=\var\{-\log(U)\}
\end{align*}
In the second line, the fact that the expected squared distance from the mean is smaller than from any other point is used, and in the fourth we used the fact that $(\log(x)+1)^2$ is convex.
\end{proof}

\begin{proof}[Proof of Theorem \ref{thm:fisher_finite}]
Let $G_n = -2 \sum \log(X_i)$. Since $U_i/2 \leq_{st} X_i$, for $i = 1, \ldots, n$, where $U_1, \ldots, U_n$ are independent uniform random variables on $[0,1]$. This implies $-\log(X_i) \leq_{st} -\log(U_i/2)$. Because the usual stochastic order is closed under convolution \citep[Theorem 1.A.3]{shaked07}, we have $G_n \leq_{st} -2 \sum \log(U_i) + 2 n \log 2$. The sum $-2 \sum \log(U_i)$ has a chi-square distribution with $2 n$ degrees of freedom, proving the first bound.

Lemma \ref{lem:mean_and_var} implies $\E(G_n) \leq 2n$ and $\var(G_n) \leq 4n$. Therefore, using Cantelli's inequality,
\begin{align*}
\Prob[G_n \geq x] & \leq \var(G_n)/\left[\var(G_n) + \{x-\E(G_n)\}^2\right]\\
& \leq \var(G_n)/\left[\var(G_n) + \{x-2n\}^2\right]\\
& \leq n/\left[n + \{(x-2n)/2\}^2\right],
\end{align*}
for $x \geq 2 n$. This proves the second bound. Finally, the moment generating function of $G_n$ is $\E\{\exp(t G_n)\} = \prod \E(X_i^{-2t})$ for $t \geq 0$. For $t \in [0,1/2)$ each $\E(X_i^{-2t}) \leq \E(U^{-2t}) = (1-2t)^{-1}$ since $x^{-2t}$ is a convex function in $x$ for $x \in [0,1]$. 
Using Markov's inequality,
\begin{align*}
\Prob(G_n\geq x) &= \Prob\{\exp(t G_n) \geq \exp(t x)\}\\
& \leq \exp(-t x)\E\{\exp(t G_n)\} \\
& \leq \exp(-t x - n \log(1-2t)),
\end{align*}
for $t \in [0,1/2)$. The minimum of this function is at $t = 1/2-n/x$, giving the third bound.
\end{proof}

\begin{proof}[Proof of Theorem \ref{thm:fisher_asymp}]
Let $V_i = -2\log(X_i)$, $\mu_V = \E(V_i)$, $W_i = -2 \log(U_i)$, where $U_1, \ldots, U_n$ are independent uniform random variables on $[0,1]$, and $\mu_W=\E(W_i)$. If $\mu_V = \mu_W$ then by Lemma~\ref{lem:mean_and_var} the $X_i$ are uniform on $[0,1]$ and we are done. The statement is also true if $\alpha=1$. Therefore assume $\mu_V < \mu_W$, $\alpha \in (0,1)$ and let $t \in (\mu_V, \mu_W)$. By the weak law of large numbers there exists an $N' \in \N$ such that, for $n \geq N'$,
\[\Prob\left(\sum_{i=1}^n W_i \geq n t\right) \geq \alpha,\]
so that $t_{\alpha,n} \geq nt$. Therefore, for $n \geq N'$,
\begin{align*}
\Prob\left(\sum_{i=1}^n V_i \geq t_{\alpha,n}\right) \leq \Prob\left(\sum_{i=1}^n V_i \geq n t\right).
\end{align*}
Again by the law of large numbers, the right-hand side tends to zero. Hence there exists an $N \geq N'$ such that it is bounded by $\alpha$ for $n \geq N$.
\end{proof}

\bibliographystyle{apalike} 
\bibliography{discrete_pvalues_jasa}

\end{document}